\documentclass[11pt]{amsart}
\usepackage{amsmath, amsthm, amsfonts, amssymb}
\usepackage{enumerate}
\usepackage[bookmarks=false]{hyperref}

\usepackage{xcolor}

\usepackage{verbatim}

\title{Conjugacy co-amenability}

\author{Mehrdad Kalantar}
\address{Mehrdad Kalantar\\University of Oxford, UK}
\email{mehrdad.kalantar@maths.ox.ac.uk}

\author{Srivatsav Kunnawalkam Elayavalli}
\address{Srivatsav Kunnawalkam Elayavalli\\University of Maryland, USA}
\email{sriva@umd.edu}

\newtheorem{thm}{Theorem}[section]
\newtheorem{prop}[thm]{Proposition}

\newtheorem{lem}[thm]{Lemma}
\theoremstyle{definition}
\newtheorem{defn}[thm]{Definition}
\newtheorem{defn/lem}[thm]{Definition/Lemma}

\newtheorem{examp}[thm]{Example}

\newcommand{\bA}{{\mathbb A}}
\newcommand{\bN}{{\mathbb N}}

\newcommand{\cX}{{\mathcal X}}

\newcommand{\eps}{\varepsilon}

\newcommand{\act}{\curvearrowright}

\newcommand{\ds}{{\sharp\kern-.5pt\sharp}}

\newcommand{\RN}[1]{%
  \textup{\uppercase\expandafter{\romannumeral#1}}%
}

\makeatletter
\DeclareRobustCommand\frownotimes{\mathbin{\mathpalette\frown@otimes\relax}}
\newcommand{\frown@otimes}[2]{%
  \vbox{
    \ialign{##\cr
      \hidewidth$\m@th#1{}_\frown$\kern-\scriptspace\hidewidth\cr
      \noalign{\nointerlineskip\kern-1pt}
      $\m@th#1\otimes$\cr
    }%
  }%
}
\makeatother

\begin{document}
\begin{abstract}
In this note we study a natural analytic property of inclusions of groups akin to co-amenability: the property of existence of a non-compactly supported invariant state for the conjugation action of a group $G$ on the von Neumann algebra generated by the characteristic functions $\{\mathbf{1}_{gHg^{-1}}\}_{g\in G}$ viewed inside $\ell^\infty(G)$.  Some interesting settings and examples of this phenomena are proved. We also comment on a consideration related to proper proximality, which motivated this property.  

\end{abstract}

\maketitle

\section{Introduction}

A group action on a set is said to be \emph{amenable} if it admits an invariant mean. The study of amenability of group actions is an important avenue in the modern field of analytic group theory. A countable group $G$ naturally acts on the coset space $G/H$ for any subgroup $H<G$. The amenability of this action is equivalent to the notion of \emph{co-amenability} of the subgroup $H<G$. This was first introduced by Eymard in \cite{Eymard}. The study of this notion turns out to be much more subtle than the study of amenability itself, permitting rather puzzling examples \cite{monodpopa}. 

We consider here a natural property of inclusions $H<G$ akin to co-amenability, which takes into account the structure of conjugates of the subgroup. We call this property \emph{conjugacy co-amenability}. In order to define it, first we must recall that the von Neumann algebra generated by a subset of $\ell^\infty(G)$: it is the closure in the strong operator topology of the adjoint-closed subalgebra generated by the subset (not necessarily having the identity element) inside $\ell^\infty(G)$. Also recall a state on a unital $C^*$-algebra is a positive linear functional that assigns the value of 1 to the identity operator. Suppose $H<G$, we define the von Neumann algebra $\mathcal{A}_H(G)$ as the von Neumann algebra generated by the family of projections $\mathbf{1}_{gHg^{-1}}$ where $g\in G$.  Note that $G$ acts naturally on $\mathcal{A}_H(G)$ by conjugation. We denote by $c_0(G)$ the ideal in $\ell^\infty(G)$ consisting of those functions $f\in \ell^\infty(G)$ such that for every $\epsilon>0$ there exists a finite subset $F\subset G$ such that $|f(g)|<\epsilon$ for every $g\notin F$. 

Now we define our notion below: 

\begin{defn}\label{Defn 1}
Let $H<G$ be an inclusion of countable groups. Say that $H<G$ is \emph{conjugacy co-amenable} if there exists a $G$-conjugation invariant state $\varphi$ on $\mathcal{A}_H(G)$ such that $\varphi(f)=0$ for all $f\in c_0(G)\cap \mathcal{A}_H(G)$. 
\end{defn}

The reason for asking the state to vanish on $c_0(G)$ is just a matter of convention, often one that is preferred to have from an operator algebraic point of view. Of course, not including this condition would yield a necessarily different notion. Naturally, whenever $H<G$ is a normal subgroup, the space $\mathcal{A}_H(G)$ is the atom at the identity element. Hence such $H<G$ is conjugacy co-amenable. In fact, more generally $H<G$ is conjugacy co-amenable whenever the normalizer of $H$ is co-amenable in $G$ (see Example \ref{normalizer example}). The converse however is not true. We provide various more interesting examples of non conjugacy co-amenability in Section \ref{non examples section}. 

Our motivating reason to define and study this notion was the problem of determining whether acylindrically hyperbolic groups are properly proximal (see \cite{BIP18, DKEP21}). An emblematic test case is amalgamated free products $G= G_1*_HG_2$, whenever $H$ is a weakly malnormal subgroup of $G$. In the case of almost malnormality, this was proved in \cite{DKE21} (see also \cite{DKE24}). In the effort to generalize the upgrading technique therein (see Lemma 3.3 of \cite{DKE21}) to the more general setting of weakly malnormal subgroups, the notion of conjugacy co-amenability naturally appeared. Interestingly, we were able to handle the part of the argument involving demonstrating non conjugacy co-amenability in these examples (see Example \ref{aylindrical example}). We however were not able to obtain a natural map to the bidual of the small at infinity compactification. Our approach is explained in Section \ref{pp section}.

\subsection*{Acknowledgments} This work was done during visits of the second author to the University of Houston in 2021. He thanks the first author, UH and its staff for their hospitality. We thank D. Osin and B. Sun for helpful exchanges concerning Example 2.12. We also thank C. Ding and F. Fournier-Facio for helpful exchanges. 
\\
The first author was supported by the NSF grant DMS-2155162; the second author is supported by the NSF grant DMS-2350049.
For the purpose of open access, the authors have applied a CC-BY license to any author accepted manuscript arising from this submission.

\section{Conjugacy co-amenability}

Let $G$ be a countable group. For $S\subset \ell^\infty(G)$ denote by $W^*(S)\subset \ell^\infty(G)$ the strong operator topology closure of the $*$-algebra generated by $S$. Note that $W^*(S)$ need not have the same identity element as $\ell^\infty(G)$.

If $H<G$ is a subgroup, then denote by $\mathcal{W}(H)$ the family of subgroups that are arbitrary intersections of conjugates of $H$. For every element $X\in \mathcal{W}(H)$, define $\widetilde{X}= X\setminus \bigcup_{Y\in \mathcal{W}(H), Y\subsetneq X} Y$.  

\begin{lem}\label{lemma1}
The family $\widetilde{\cX}:=\{\widetilde{X} : X\in \mathcal{W}(H)\}$ is pairwise disjoint, and $X = \bigcup_{Y\in \mathcal{W}(H), Y\subseteq X}\widetilde{Y}$ for every $X\in \mathcal{W}(H)$.
\begin{proof}
Assume $\widetilde{X_1}\cap \widetilde{X_2} \neq\emptyset$ for some $X_1, X_2\in \mathcal{W}(H)$, and let $g\in \widetilde{X_1}\cap \widetilde{X_2}$. Then $g\in {X_1}\cap {X_2}$, and so it follows ${X_1}\cap {X_2}= X_1$ and ${X_1}\cap {X_2}= X_2$. Thus $X_1=X_2$.

To show the second assertion, let $g\in X$, and let $K:= \bigcap_{Y\in \mathcal{W}(H),\, g\in Y\subseteq X}Y$. Then $K\in \mathcal{W}(H)$, and for every $Y\in \mathcal{W}(H)$ with $Y\subsetneq K$ we have $g\notin Y$. Hence $g\in \widetilde{K}$.
\end{proof}
\end{lem}

\begin{lem}
For every $g\in G$ and $X\in \mathcal{W}(H)$ we have $\widetilde{X^g} = \widetilde{X}^g$. In particular, $G$ acts on the set $\widetilde{\cX}$ by conjugations.
\begin{proof}
This follows immediately from the facts that the family $\mathcal{W}(H)$ is invariant under conjugations by elements in $G$, and that for every $X, Y\in \mathcal{W}(H)$ and $g\in G$ we have $Y\subsetneq X \iff Y^g\subsetneq X^g$.
\end{proof}
\end{lem} 

Observe that $\mathcal{A}_H(G)= W^*(\{\mathbf{1}_{\widetilde{X}}\}_{X\in \mathcal{W}(H)}) = \ell^\infty(\widetilde{\cX})$, and $\widetilde{\cX}$ decomposes into a countable disjoint union of transitive $G$-spaces. Thus, there is $\bA\subseteq \bN$, and subgroups $\{L_n\le G :n\in \bA\}$ such that $\widetilde{\cX} = \bigsqcup_{n\in \bA} G/L_n$. 

In particular, if $L_n$ is co-amenable in $G$ for some $n\in\bA$, then $H<G$ is conjugacy co-amenable in $G$.

We denote $G_{\widetilde{X}} = \{g\in G : g\widetilde{X} = \widetilde{X}\}$ for the stabilizer subgroup of $\widetilde{X}\in \widetilde{\cX}$, and we denote by $\mathcal{N}_G(X)$ the normalizer in $G$ of a subgroup $X\le G$.

\begin{lem}\label{lem:stab}
We have $G_{\widetilde{X}} = \mathcal{N}_G(X)$ for every $X\in \mathcal{W}(H)$ with $\widetilde{X}\neq\emptyset$.

In particular, if the normalizer $\mathcal{N}_G(H)$ is co-amenable in $G$, then $H$ is conjugacy co-amenable in $G$.
\begin{proof}
Let $X\in \mathcal{W}(H)$ and $g\in \mathcal{N}_G(X)$. Then $g\widetilde{X}g^{-1} \subset X$. If $g\widetilde{X}g^{-1} \cap Y \neq \emptyset$ for some $Y\in \mathcal{W}(H)$ with $Y\subsetneq X$, then $\widetilde{X} \cap g^{-1}Y g \neq \emptyset$, which implies $\widetilde{X} \subseteq g^{-1}Y g$, thus $X= g\widetilde{X}g^{-1} \subseteq Y\subsetneq X$, which is a contradiction. Hence, $g\widetilde{X}g^{-1} =\widetilde{X}$, and so $g\in G_{\widetilde{X}}$.

Conversely, assume $g\in G_{\widetilde{X}}$. Then $\widetilde{X}=g\widetilde{X}g^{-1} \subset X\cap g{X}g^{-1}$, which implies $X \subset g{X}g^{-1}$. Similarly, $X \subset g^{-1}{X}g$, hence $X = g{X}g^{-1}$, and so $g\in \mathcal{N}_G(X)$.
\end{proof}
\end{lem}

We will need the following immediate auxiliary lemma. A subset $I\subset K<G$ is said to be malnormal if $gxg^{-1}\notin K$ for all $x\in I, g\notin K$. 

\begin{lem}
    $X\in \mathcal{W}(H)$ contains an infinite malnormal subset $I\subset X$ iff $|\widetilde{X}|=\infty$. 
\end{lem}

Recall that an action $G\act S$ of a group $G$ on a set $S$ is said to be amenable if there is a $G$-invariant state on $\ell^\infty(S)$. It is standard fact that this is equivalent to existence of a net $(\zeta_j)\in \ell^2(S)$ such that $\lim_j\|g\zeta_j-\zeta_j\|_2 = 0$ for every $g\in G$.

\begin{prop}
Let $G$ be a group and $\{H_i :i\in I\}$ be a family of subgroups of $G$. Then the following are equivalent:
\begin{enumerate}
\item
The action $G\act \sqcup_{i\in I} G/H_i$ is amenable.
\item
For every finite set $F\subset G$ and every $\eps>0$, there exists $i\in I$ and $\xi_i\in \ell^2(G/H_i)$ such that $\|g\xi_i-\xi_i\|_2\le \eps\|\xi_i\|$ for every $g\in F$.
\end{enumerate}
\end{prop}

\begin{proof}
Assume (2) does not hold. Then there exist a finite set $F\subset G$ and $\eps>0$ such that for every $i\in I$ and every $\xi_i\in \ell^2(G/H_i)$ we have $\|g\xi_i-\xi_i\|_2 > \eps\|\xi_i\|$ for some $g\in F$.
Given $\eta= (\eta_i)_{i\in I}\in \oplus_{i\in I} \ell^2(G/H_i) =  \ell^2(\sqcup_{i\in I} G/H_i)$ we have
\[\begin{split}
\sum_{g\in F}\|g\eta-\eta\|_2^2 &= \sum_{g\in F}\sum_{i\in I}\|g\eta_i-\eta_i\|_2^2 \\&= \sum_{i\in I}\sum_{g\in F}\|g\eta_i-\eta_i\|_2^2 \\&> \eps^2\,\sum_{i\in I}\|\eta_i\|^2 = \eps^2\,\|\eta\|^2 ,
\end{split}
\]
which implies that $\|g\eta-\eta\|_2 > \eps/|F|$ for some $g\in F$. This implies that the action $G\act \sqcup_{i\in I} G/H_i$ is not amenable.

Conversely, assume (2) holds. Then, considering the canonical embedding $\ell^2(G/H_{i_0}) \subset\oplus_{i\in I} \ell^2(G/H_i)$ for every $i_0\in I$, we conclude that for every finite set $F\subset G$ and every $\eps>0$, there exists $\xi_{F,\eps}\in \oplus_{i\in I} \ell^2(G/H_i) =  \ell^2(\sqcup_{i\in I} G/H_i)$ such that $\|g\xi_{F,\eps}-\xi_{F,\eps}\|_2\le \eps\|\xi_{F,\eps}\|$ for every $g\in F$. Obviously the net $(\xi_{F,\eps})\in \ell^2(\sqcup_{i\in I} G/H_i)$ satisfies $\|g\xi_{F,\eps}-\xi_{F,\eps}\|_2\to 0$ for every $g\in F$. It follows that the action $G\act \sqcup_{i\in I} G/H_i$ is amenable.
\end{proof}

\begin{lem}
Let $G$ be a group and $H\le G$. If $G$ has a non-amenable subgroup $K\le G$ such that $K\cap gHg^{-1} = \{e\}$ for all $g\in G$, then $H$ is not co-amenable in $G$.
\end{lem}

\begin{proof}
Since $K$ is non-amenable, there exists a finite set $F\subset K$ and $\eps>0$ such that $\max_{t\in F} \|t\xi - \xi\|_2 \ge \eps\|\xi\|$ for every $\xi\in\ell^2(K)$. 

Let $S\subset G$ be a complete set of representatives for the double coset space $K/G\backslash H$ (that is, $G = \sqcup_{s\in S} KsH$). Note that for $k, k'\in K$ and $s, s'\in S$, if $ksH = k's'H$ then $s=s'$ and $s^{-1}k^{-1}k's\in H$, which implies $k=k'$ since $K\cap sHs^{-1} = \{e\}$.

Given $\eta\in\ell^2(G/H)$, for every $s\in S$ we define $\eta_s\in\ell^2(K)$ by $\eta_s(k):= \eta(ks)$ for every $k\in K$. 
 Then
\[
\max_{t\in F} \|t\eta - \eta\|_2^2 = \max_{t\in F} \sum_{s\in S}\|t\eta_s - \eta_s\|_2^2 \ge \eps^2\sum_{s\in S}\|\eta_s\|_2^2 =\eps^2\|\eta\|_2^2 .
\]
This shows that $H$ is not co-amenable in $G$.
\end{proof}
It may worth noting that in the above proof we only used $K\cap sHs^{-1} = \{e\}$ for all $s\in S$. The same proof shows the following: 

\begin{lem}\label{lem:ncca}
Let $G$ be a group and $\{H_i :i\in I\}$ be a family of subgroups of $G$. If $G$ has a non-amenable subgroup $K\le G$ such that $K\cap gH_ig^{-1} = \{e\}$ for all $g\in G$, $i\in I$, then the action $G\act \sqcup_{i\in I} G/H_i$ is not amenable.
\end{lem}

We now obtain the following useful condition for verifying the negation of conjugacy co-amenability: 

\begin{prop}\label{useful negatation}
Let $H<G$, and assume that there exists a non-amenable subgroup $K<G$ such that $gkg^{-1} H (gkg^{-1})^{-1}\cap H= \{e\}$ for every $g\in G$ and $1\neq k\in K$. Then $H$ is not conjugacy co-amenable in $G$.
\end{prop}
\begin{proof}
For every $X\in \mathcal{W}(H)$, we have $gkg^{-1}\notin \mathcal{N}_G(X)$ by the assumption and the fact that $X<g_0Hg_0^{-1}$  for some $g_0\in G$. Thus, the result follows from Lemma~\ref{lem:ncca}.  
\end{proof}

\subsection{Some (non)examples}\label{non examples section}

In this section we give several examples of (non) conjugacy co-amenable subgroups.

\begin{examp}
If $H<G$ is an infinite normal subgroup, then one always has $H<G$ is conjugacy co-amenable since $\mathcal{A}_H(G)$ is just the atom at the identity element. 
\end{examp}

\begin{examp} \label{normalizer example}
As noted in Lemma~\ref{lem:stab}, if $H<G$ is such that the normalizer $\mathcal{N}_G(H)$ is co-amenable in $G$, then $H$ is conjugacy co-amenable in $G$. The converse, however, is not true. Indeed, let $H<G$ be such that its normalizer $\mathcal{N}_G(H)$ is not co-amenable in $G$ and $N:=\bigcap_{g\in G}gHg^{-1}$ is infinite. Then $H$ is conjugacy co-amenable in $G$, since in this case $\widetilde{N}=N\in \widetilde{\mathcal{X}}$ is a $G$-fixed point, and so $\delta_N$ is a $G$-invariant state on $\mathcal{A}_H(G)$ which vanishes on $c_0(G)$ since $N$ is infinite.  
An example of such $H<G$ is
$G = F_2\times F_2$ and $H=\mathbb{Z}\times F_2$. 
\end{examp}

\begin{examp} Let $H<G$ be almost malnormal in $G$. In this case, suppose $H<G$ is conjugacy co-amenable then there is a $G$ invariant state on the quotient space $\mathcal{A}_H(G)/c_0(G)$. Almost malnormality says that $gHg^{-1}\cap H$ is finite for any $g\in G$, $g\notin H$. Hence the conjugacy action of $G$ on $\mathcal{A}_H(G)/c_0(G)$ is equivariantly identified with the coset space $\ell^\infty(G/H)$ with the left action of $G$. This would therefore imply that $H<G$ is co-amenable. Note that this is impossible from Lemma 3.3 of \cite{TDDuke}.
\end{examp}

\begin{examp} Assume that $H<G$ is finite height in $G$. By definition there is an $N>0$ such that $\cap_{i=1}^{N}H^{g_i}$ is finite for all distinct $N$-tuples $g_i\in G/H$. Define the  maps  $\pi_i:\ell^\infty(G/H)\to \mathcal{A}_H(G)/c_0(G)$ where $i=1,\cdots, N-1$ in the following way: $\pi_{N-1}(\delta_g)= \sum_{\{g_1,\hdots g_{N-1}| g_i\neq g\}}\prod_{i=1}^{N-1} \mathbf{1}_{H^{g_i}}$, and ${\pi_j(\delta_g)}= \sum_{g_1,\cdots, g_{j}| g_i\neq g}\prod_{i=1}^{j}\left(\mathbf{1}_{H^{g_i}}-\mathbf{1}_{H^{g_i}} (\bigvee_{k=j+1}^{N-1}\pi_{k}(1))\right)$. The maps then can be extended as follows: $\pi_j(\sum_{g\in G/H}\lambda_g \delta_g)= \bigvee_{g\in G/H} \lambda_g \pi_j(\delta_g)$. 

We claim that $\pi_i$ as defined above are all $G$-equivariant contractive injective $*$-homomorphisms. First observe that the projections  $$\prod_{i=1}^{j}\left(\mathbf{1}_{H^{g_i}}-\mathbf{1}_{H^{g_i}} (\bigvee_{k=j+1}^{N-1}\pi_{k}(1))\right)$$ from above are all pairwise orthogonal by construction for any $j$. Moreover, by construction $\bigvee_{j=1}^{N-1}\pi_j(1)=1$, the identity element of $\mathcal{A}_H(G)/c_0(G)$. Hence we see that $H<G$ is conjugacy co-amenable only if $H<G$ is co-amenable. 
\end{examp}

The relevance of property (T) groups $G$ in this context is obvious. By the spectral gap, a subgroup $H<G$ is conjugacy co-amenable if and only if for some $X\in \mathcal{W}(H)$ with $\tilde X \neq \emptyset$, its normalizer $\mathcal{N}_G(X)$ has finite index in $G$. This fact provides large classes of examples of non conjugacy co-amenable subgroups $H<G$.

\begin{examp}
Let $G$ be a property (T) group and let $H<G$.
\\
(i) Assume $G$ has trivial amenable radical and $H$ is amenable. Then $H$ is not conjugacy co-amenable in $G$. Indeed, all finite index subgroups of $G$ have trivial amenable radical, hence for every non-trivial $X\in \mathcal{W}(H)$, the normalizer $\mathcal{N}_G(X)$ has infinite index in $G$.
\\
(ii) Assume $H$ is self-normalizing and has infinite index in $G$, then $H$ is not conjugacy co-amenable in $G$. Examples of such $H<G$ include $G={\rm PSL}_n(\mathbb Z)$ for $n\ge 3$, and $H$ the standard copy of ${\rm PSL}_{n-1}(\mathbb Z)$.
\end{examp}

\begin{examp}
Let $G$ be a non-amenable group with the property that the normalizer of every non-trivial amenable subgroup of $G$ is also amenable. Then every non-trivial amenable subgroup $H<G$ is not conjugacy co-amenable in $G$.
\end{examp}

The examples of groups $G$ satisfying the condition of the last example include all torsion-free hyperbolic groups. In connection with the following example which generalizes the above case, the second author acknowledges helpful exchanges with D. Osin from some years ago, and B. Sun who kindly explained the following argument.

\begin{examp}\label{aylindrical example} Consider $H<G$ where $H$ is an arbitrary edge stabilizer of a torsion-free group $G$ acting non elementarily and edge-transitively on a tree admitting a loxodromic WPD-element in the sense of \cite{Bestvina_2002}. We claim that $H<G$ is not conjugacy co-amenable. Towards this it suffices to prove using Proposition \ref{useful negatation} that there exists a free group $F_2\leq G$ such that for all edges $e\in E(T)$ and all $1\neq f\in F_2$, we have $Stab(f\cdot e)\cap Stab(e)=1$.

Now suppose $a$ is the loxodromic WPD element, and let $a^{\pm \infty}$ denote the fixed points of $a$ on $\partial T$. Since we assumed that the action is non-elementary, there exists an element $g\in G$ that does not fix $\{a^{\pm \infty}\}$. Now fix $c=gag^{-1}$ and pick $x,y\in V(T)$ such  that $d(x,ax)\leq d(z,az)$ and $d(y,cy)\leq d(z,cz)$ for every $z\in T$. We can assume without loss of generality that there is some $k\in G$ such that $ky=x$. Setting $b=k^{-1}ck$ yields two loxodromic WPD elements $a,b$ whose  fixed point sets are disjoint in $\partial T$. By a standard argument, one can take large powers $m,n$ such that $a^m, b^n$ generate a free group $F_2$, and also satisfying the WPD condition, i.e, $Stab(x, a^{\pm m}x)= Stab(x, b^{\pm n}x)=1$. Set $\pi$ to be the projection of $T$ to the subtree consisting of all geodesics between points in the orbit  $F_2.x$.
Now pick any edge $e\in E(T)$, and $z\in e$ and by translation we may assume that $\pi(z)=x$. Suppose $f\in G_2$ and $t\in Stab(e)\cap Stab(fe)$. Since the action is by isometries, we have that the whole geodesic $[z,fz]$ has to be fixed. In particular, we have that $[\pi(z), \pi(fz)]=[x,fx]$ is fixed. But this contains a segment of the form $[x,a^{\pm m}x]$ or $[x, b^{\pm n}x]$ which contradicts the WPD assumption.

A very natural example of the above family of examples via Bass-Serre theory is the class of groups $G= G_1*_H G_2$ where $H$ is weakly malnormal, i.e, there exists an element $g\in G$ such that $gHg^{-1}\cap H$ is trivial. (See \cite{MO15} for the above result, and additional information about such groups.) Remarkably it is still open whether such groups are properly proximal in the sense of \cite{BIP18}. The following section outlines an approach to solve this problem.
\end{examp}

\section{A comment on proper proximality}\label{pp section}

Let $G$ be a countable group.  Let $K_i<G$, $i=1,\hdots, n$ be a finite family of subgroups of $G$. Denote by $c_0(G,\{K_i\}_{i=1}^{n})$ the set of functions $f\in \ell^\infty(G)$ such that for all $\epsilon>0$ there exists a finite set $F\subset G$ satisfying $|f(g)|<\epsilon$ for all $g\notin \bigcup_{i=1}^{n}FK_iF$. It is easy to check that this is always a closed 2-sided ideal in $\ell^\infty(G)$. Let $\mathbb{S}_{\{K_{i}\}_{i=1}^{n}}(G)=\left(\ell^{\infty}(G)/c_0(G,\{K_i\}_{i=1}^{n})\right)^{R_{G}}$ be the fixed points under the right action of $G$ on the relative Stone-Cech boundary $\ell^{\infty}(G)/c_0(G,\{K_i\}_{i=1}^{n})$. Recall the following  bidual version of the relative small boundary $\widetilde{\mathbb{S}}_{\{K_{i}\}_{i=1}^{n}}(G)= \left(\left(\ell^{\infty}(G)/c_0(G,\{K_i\}_{i=1}^{n})\right)^{**}\right)^{R_{G}}$. As in \cite{BIP18}, $G$ is properly proximal relative to a family of subgroups $\{K_i\}_{i=1}^{n}$ if there is no left $G$-invariant state on $\mathbb{S}_{\{K_{i}\}_{i=1}^{n}}(G)$. Equivalently by Theorem 4.3 (iv) of \cite{BIP18},   $G$ is properly proximal relative to a family of subgroups $\{K_i\}_{i=1}^{n}$ if and only if there is no left $G$-invariant state on $\widetilde{\mathbb{S}}_{\{K_{i}\}_{i=1}^{n}}(G)$.

Given a subgroup $H<G$, denote by $\iota$ the canonical injective contractive $*$-algebra homomorphism from $\ell^\infty(G)/c_0(G)$  into the Banach space bidual $\left(\ell^\infty(G)/c_0(G)\right)^{**}$. For any subset $X\subset G$, define $p_X= \sup_{g\in G} \iota (\mathbf{1}_{Xg})$, also called the right averaging projection of $X$. Define $P_X= \sup_{g,h\in G} \iota (\mathbf{1}_{gXh})$ also called the left and right averaging projection of $X$. The following is a  key lemma that is inspired by the work \cite{DKE21}.  

\begin{lem}\label{intersection}
For any two subgroups $X_1< G$, and $X_2< G$, we have that $p_{X_1}p_{X_2}= p_{X_1\cap X_2}$.       
\end{lem}

\begin{proof}
It is obvious from the definition that $p_{X_1\cap X_2}\le p_{X_1}p_{X_2}$.
Now, for any $g_1, g_2\in G$, we see that if $t\in X_1g_1\cap X_2g_2$, then $X_1g_1\cap X_2g_2= (X_1\cap X_2)t$.
Thus, given a finite set $F\subset G$, there is a finite set $F^\prime\subset G$ such that
$H_1F\cap H_2F= \bigcup_{g_1,g_2\in F} X_1g_1 \bigcup  X_2g_2= \bigcup_{t\in F'} (X_1\cap  X_2)t$, and therefore, $p_{X_1}p_{X_2}= \bigvee_{F\subset G} \iota(\mathbf{1}_{X_1F\cap X_2F})\leq p_{X_1\cap X_2}$.
\end{proof}

Let $\mathcal{A}_{H}^{1}(G)$ denote the $C^*$-algebra generated by $\{ \mathbf{1}_{gHg^{-1}}\}_{g\in G}\subset \ell^\infty(G)$.

\begin{lem}\label{pi extends}
For any pair of distinct elements $X,Y\in \mathcal{W}(H)$, we have $p_{\widetilde{X}}p_{\widetilde{Y}}=0$.
Consequently, the map $\mathbf{1}_{gHg^{-1}}\mapsto p_{gHg^{-1}}$ extends (uniquely) to an injective $*$-homomorphism $\pi: \mathcal{A}^{1}_H(G)/c_0(G)\to P_H\widetilde{\mathbb{S}}(G)P_H$.
\end{lem}
\begin{proof}
Let $X,Y\in \mathcal{W}(H)$ with $X\neq Y$. Then for any finite set $F\subset G$, we get  $\mathbf{1}_{\widetilde{X}F}\leq \mathbf{1}_{(X\setminus X\cap Y)F}=  \mathbf{1}_{XF}- \mathbf{1}_{(X\cap Y)F}$. Therefore, we see $p_{\widetilde{X}}\leq p_{X}- p_{X\cap Y}$.    Then by Lemma \ref{intersection} we have $p_{\widetilde{X}}p_{\widetilde{Y}}\leq (p_{X}-p_{X}p_{Y})(p_{Y})= 0$ as required.

It follows that the linear extension of the map $\mathbf{1}_{gHg^{-1}}\mapsto p_{gHg^{-1}}$ to the linear span of $\{\mathbf{1}_{gHg^{-1}}: g\in G\}$ is well-defined, isometric, and $*$-homomorphism, hence extends to a $C^*$-embedding $\pi: \mathcal{A}^{1}_H(G)/c_0(G)\to P_H\widetilde{\mathbb{S}}(G)P_H$.
It is obvious that $\pi$ is $G$-equivariant.
\end{proof}

Recall the von Neumann algebra $\mathcal{A}_H(G)= W^*(\{\mathbf{1}_{gHg^{-1}}\}_{g\in G})$. For notational simplicity, denote by $\mathbf{1}$ the identity element of $\mathcal{A}_H(G)$.  A natural way to construct a map $\pi: \mathcal{A}_H(G)\to
\left(\left(\ell^\infty(G)/c_0(G)\right)^{**}\right)^{R_G}$ such that $\pi(\mathbf{1})= P_{H}$ would be the following:

\begin{enumerate}
    \item $\pi(\mathbf{1}_{gHg^{-1}})= p_{g}$;
    \item For $X\in  \mathcal{W}(H)$, and $X=\cap_{i\in I_X} g_iHg_i^{-1}$, we define $\pi(\mathbf{1}_{X})=  \land_{i\in I} p_{g_i}$. 

    \item Define $\pi(\mathbf{1}_{\bigcup_{Y\in \mathcal{W}(H), Y\subsetneq X} Y})= \bigvee_{Y\in \mathcal{W}(H), Y\subsetneq X} p_{Y}$.
    \item For $\widetilde{X}= X\setminus \bigcup_{Y\in \mathcal{W}(H), Y\subsetneq X} Y$, define $\pi(\mathbf{1}_{\widetilde{X}})= \pi(\mathbf{1}_{X})-\pi(\mathbf{1}_{\bigcup_{Y\in \mathcal{W}(H), Y\subsetneq X} Y})$.
    \item $\pi(\bigsqcup_{\widetilde{X}\in S\subset \widetilde{\cX}} \mathbf{1}_{\widetilde{X}})= \bigvee \pi(\mathbf{1}_{\widetilde{X}})$. 
\end{enumerate}

We do have the following:
\begin{lem}
For every $X\in \mathcal{W}(H)$, we have $\bigvee_{F\subset G} \iota(\mathbf 1_{\widetilde{X}F})\leq \pi(1_{\widetilde{X}})$.
\end{lem}

\begin{proof}
Fix $X= \cap_{i=1}^\infty g_iHg_i^{-1}$ where $g_i\in G$. For every finite subset $F\subset G$ and every $Y\in \mathcal{W}(H)$ that is a finite intersection of conjugates of $H$, and every $n\in \mathbb{N}$, we have  $$\iota(\mathbf 1_{\widetilde{X}F})\leq   \iota(\mathbf 1_{(\cap_{i\leq n}g_iHg_i^{-1}\setminus Y )F})\leq \iota(\mathbf 1_{(\cap_{i\leq n}g_iHg_i^{-1})F}) - \iota(\mathbf 1_{(Y\cap (\cap_{i\leq n}g_iHg_i^{-1}))F}) .$$ Applying Lemma \ref{intersection}, and continuing the above chain of inequalities we see that 
\[
\begin{split}
\iota(\mathbf 1_{\widetilde{X}F})&\leq \bigvee_{F\subset G}\left(  \iota(\mathbf 1_{(\cap_{i\leq n}g_iHg_i^{-1})F}) - \iota(\mathbf 1_{(Y\cap (\cap_{i\leq n}g_iHg_i^{-1}))F})\right)
\\&= \land_{i\leq n}p_{g_i} - (\land_{i\leq n}p_{g_i})p_{Y}.
\end{split}
\]
Hence, taking inf over $n$ and sup over $Y$, we get  $\iota(\mathbf 1_{\widetilde{X}F})\leq \pi(\widetilde{X})$, and the result follows.   
\end{proof}

From the construction of $\pi$ it is immediate that the projection $\pi(\mathbf{1}_{\widetilde{X}})$ is orthogonal to $\pi(\mathbf{1}_{\widetilde{Y}})$ for every $\widetilde{Y}\neq \widetilde{X}$. The main issue that kills this approach is the problem of verifying that $\pi(1)=P_H$. If such a map $\pi$ existed and it is proven to be positive, then one can prove proper proximality for new families of groups, notably including amalgamated free products $G*_AH$ where $A$ is weakly malnormal  (see Example \ref{aylindrical example}). Indeed, recall that
if $G$ is properly proximal relative to $H$, and $G$ is not properly proximal, then there exists a left $G$ invariant state $\phi$ on $\widetilde{\mathbb{S}}(G)$ such that $\phi(P_H)=1$. Moreover, assume $G$ is properly proximal relative to $H$, and $H$ is not conjugacy co-amenable in $G$. Then $G$ is properly proximal. Indeed, assume that $G$ is not properly proximal by contradiction. There exists a left $G$ invariant state $\phi$ on $\widetilde{\mathbb{S}}(G)$ such that $\phi(P_H)=1$.   Then by the above, we see that  $\phi\circ \pi$  is a $G$-conjugation invariant state on $\mathcal{A}_H(G)$ as required, which contradicts the non conjugacy co-amenability of $H$.

\bibliographystyle{amsalpha}
\bibliography{ref}

\end{document}